\documentclass[12pt]{amsart}
\usepackage{amssymb,mathrsfs}
\usepackage{amsthm,amsmath,amssymb}
\usepackage[numbers,sort&compress]{natbib}

\makeatletter

\newcommand{\Rmnum}[1]{\expandafter\@slowromancap\romannumeral #1@}
\makeatother
\usepackage{mathtools}

\theoremstyle{plain}
\newtheorem{theorem}{Theorem}[section]

\newtheorem{proposition}[theorem]{Proposition}
\newtheorem{corollary}[theorem]{Corollary}

\theoremstyle{definition}
\newtheorem{definition}[theorem]{Definition}

\newtheorem{example}[theorem]{Example}
\newtheorem{problem}[theorem]{Problem}

\usepackage[dvips]{graphics}
\usepackage{mathrsfs}
\usepackage{amssymb}
\usepackage{stmaryrd}
\usepackage{pifont}
\usepackage{amsmath}
\usepackage{hyperref}
\usepackage{amsthm}
\usepackage{cases}
\usepackage{mathtools}
\allowdisplaybreaks 
\title[unbounded continuous functionals]{Continuous functionals for unbounded convergence in Banach lattices}

\date{\today}

\keywords{Riesz space, Banach lattice, unbounded order convergence, unbounded norm convergence, unbounded absolute weak convergence, unbounded absolute weak* convergence.}
\subjclass[2010]{46A40, 46B42}

\author[Z. Wang]{Zhangjun Wang$^{1}$}
\address{$^1$ The first author:School of Mathematics, Southwest Jiaotong University,
	Chengdu, Sichuan,
	China, 610000.}
\email{zhangjunwang@my.swjtu.edu.cn}

\author[Z. Chen]{Zili Chen$^{2}$}
\address{$^2$ The second author:School of Mathematics, Southwest Jiaotong University, Chengdu, Sichuan,
	China, 610000.}
\email{zlchen@home.swjtu.edu.cn}

\author[J. Chen]{Jinxi Chen$^{3}$}
\address{$^3$ The third author: School of Mathematics, Southwest Jiaotong University, Chengdu, Sichuan,
	China, 610000.}
\email{jinxichen@home.swjtu.edu.cn}

\begin{document}

\begin{abstract}
Recently, the different types of unbounded convergences $(uo, un, uaw, ua\\w^*)$ in Banach lattices were studied. In this paper, we study the continuous functionals with respect to unbounded convergences. We first characterize the continuity of linear functionals for these convergences. Then we define the corresponding unbounded dual spaces and get their exact form. Based on these results, we discuss order continuity and reflexivity of Banach lattices. Some related results are obtained as well.
\end{abstract}
	
\maketitle

\section{Introduction}
A net $(x_\alpha)_{\alpha\in A}$ in a Riesz space $E$ is order convergent to $x\in E$ (write $x_\alpha\xrightarrow{o}x$) if there exists a net $(y_\beta)$, possibly over a different index set, such that $y_\beta\downarrow0$ and for each $\beta\in B$ there exists $\alpha_0\in A$ satisfying $|x_\alpha-x|\leq y_\beta$ for all $\alpha\geq \alpha_0$. The \emph{unbounded order convergence} is considered firstly by Nakano in \cite{N:48} and introduced in \cite{D:64,W:77}. A net $(x_\alpha)$ in a Banach lattice $E$ is unbounded order (resp. norm, absolute weak) convergent to some $x$, denoted by $x_\alpha\xrightarrow{uo}x$ (resp. $x_\alpha\xrightarrow{un}x$, $x_\alpha\xrightarrow{uaw}x$), if the net $(|x_\alpha-x|\wedge u)$ converges to zero in order (resp. norm, weak) for all $u\in E_+$. A net $(x_\alpha^\prime)$ in a dual Banach lattice $E^\prime$ is unbounded absolute weak* convergent to some $x^\prime$, denoted by $x_\alpha^\prime\xrightarrow{uaw^*}x^\prime$, if $|x^\prime_\alpha-x^\prime|\wedge u^\prime\xrightarrow{w^*}0$ for all $u^\prime\in E_+^\prime$. Recently, there are different kind results involved these convergence (see \cite{GTX:16,GLX:18,KMT:16,Z:16,W:19}). In \cite{GTX:16,GLX:18}, some properties of $uo$-convergence in Riesz spaces and Banach lattices be studied. For the properties of $un$, $uaw$ and $uaw^*$-convergence, we refer to \cite{KMT:16,Z:16,W:19}.

It can be easily verified that, in $l_p (1\leq p<\infty)$, $uo$, $un$ and $uaw$ and $uaw^*$-convergence of nets are the same as coordinate-wise convergence.  In $L_p(\mu) (1\leq p<\infty)$ for finite measure $\mu$, $uo$-convergence for sequences is the same as almost everywhere convergence, $un$ and $uaw$-convergence for sequences are the same as convergence in measure. In $L_p(\mu) (1< p<\infty)$ for finite measure $\mu$, $uaw^*$-convergence for sequences is also the same as convergence in measure.

In \cite{GLX:18}, Gao et al. studied the continuity of the linear functionals for $uo$-convergence. The aims of the present paper is the continuity of linear functionals for different types of unbounded convergences $(uo; un; uaw; uaw^*)$ in Banach lattices. A linear functional $f$ on a Banach lattice $E$ is said to be $uo$ (resp. $un$, $uaw$, $uaw^*$)-continuous whenever $f(x_\alpha)\rightarrow0$ for all $uo$ (resp. $un$, $uaw$, $uaw^*$)-null net $(x_\alpha)$ in $E$. In the first part of the paper, we investigate the continuity of the linear functional $f$ and prove that the carriers of $f$ are finite-dimensional. Then we assume that the net $(x_\alpha)$ is norm bounded. We characterize the continuity of such functionals and obtain the exact form of the corresponding dual spaces. As an application of these results, we conclude the paper with characterizations of the order continuity and reflexivity of Banach lattices.

Recall that a Riesz space $E$ is an ordered vector space in which $x\vee y=\sup\{x,y\}$ and $x\wedge y=\inf\{x,y\}$ exists for every $x,y\in E$. The positive cone of $E$ is denoted by $E_+$, $i.e.,E_+=\{x\in E:x\geq0\}$. For any vector $x$ in $E$ define $x^+:=x\vee0,x^-:=(-x)\vee0, |x|:=x\vee(-x)$. An operator $T:E\rightarrow F$ between two Riesz spaces is said to be \emph{positive} if $Tx\geq0$ for all $x\geq0$. A net $(x_\alpha)$ in a Riesz space is called disjoint whenever $\alpha\ne \beta$ implies $|x_\alpha|\wedge|x_\beta|=0$ (denoted by $x_\alpha\perp x_\beta$). A set $A$ in $E$ is said to be \emph{order bounded} if there exsits some $u\in E_+$ such that $|x|\leq u$ for all $x\in A$. The solid hull $Sol(A)$ of $A$ is the smallest solid set including $A$ and it equals the set $Sol(A):=\{x\in E:\exists y\in A,|x|\leq|y|\}.$ An operator $T:E\rightarrow F$ is called order bounded if it maps order bounded subsets of $E$ to order bounded subsets of $F$. A Banach lattice $E$ is a Banach space $(E, \Vert\cdot\Vert)$ such that $E$ is a Riesz space and its norm satisfies the following property: for each $x, y\in E$ with $|x|\leq|y|$, we have $\Vert x\Vert\leq\Vert y\Vert$. Recall that a vector $e > 0$ in an Banach lattice lattice $E$ is an atom if for any
$u, v \in [0, e] $ with $u \wedge v = 0$, either $u = 0$ or $v = 0$. In this case, the band generated by $e$ is $span\{e\}$. Moreover, the band projection $P_e:E\rightarrow span\{e\}$ defined by
$$P_ex = \sup_n(x^+\wedge ne)-\sup_n(x^-\wedge ne)$$
exsits, and there is a unique positive linear functional $f_e$ on $E$ such that $P_e(x) = f_e(x)e$
for all $x\in E$. We call $f_e$ the coordinate functional with the atom $e$. Clearly, the span of any finite set of atoms is also a projection band.

For undefined terminology, notation and basic theory of Riesz space, Banach lattice and linear operator, we refer to \cite{AB:06,MN:91}.

\section{Results}\label{}
    Let us determine continuous functionals with respect to unbounded
convergences on $\ell_1$. 

\begin{example}
	Let $(x_\alpha)$ be a $uo$-null, $un$-null, $uaw$-null, $uaw^*$-null and disjoint net in $\ell_1$. Clearly, $(x_\alpha)$ is coordinate-wise convergence. For a vector $\lambda=(\lambda_1,\lambda_2,...,\lambda_n,...)$ satisfying $\lambda(x_\alpha)\rightarrow0$, it can be easily verified that $\lambda\in c_{00}$.
\end{example}

According to the above example, we can find that the carriers of the $uo$-continuous, $un$-continuous, $uaw$-continuous, $uaw^*$-continuous and disjoint continuous functionals $\lambda$ on $l_1$ are finite-dimensional. It is natural to ask that whether the carriers are finite-dimensional in more general situations. The following results confirm the hypothesis.

For an operator $T : E \rightarrow F$ between two Riesz spaces
we shall say that its modulus $|T|$ exists (or that $T$ possesses a modulus)
whenever $|T| := T \vee (-T)$ exists. The carrier of $T$ is denoted by $C_T$ with $C_T:=\{x\in E:|T|(|x|)=0\}^d$.
\begin{theorem}
	Let $E$ be an atomic Banach lattice and $F$ a Banach lattices, for a nonzero linear operator $T:E\rightarrow F$, assume that the modulus $|T|$ exsits, then $C_T$ is generated by finitely many atoms, if one of the following conditions is satisfied.
	\begin{enumerate}
		\item $Tx_\alpha\rightarrow0$ for every disjoint net $(x_\alpha)\subset E$.
		\item $Tx_\alpha\rightarrow0$ for every $uo$-null net $(x_\alpha)\subset E$.
		\item $Tx_\alpha\rightarrow0$ for every $un$-null net $(x_\alpha)\subset E$ and $E$ has order continuous norm.
		\item $Tx_\alpha\rightarrow0$ for every $uaw$-null net $(x_\alpha)\subset E$.
		\item $Tx_\alpha\rightarrow0$ for every $uaw^*$-null net $(x_\alpha)\subset E$ whenver $E$ is a dual Banach lattice.
	\end{enumerate}
\end{theorem}
\begin{proof}
$(1)$ We claim that $C_T$ can not contain an infinite disjoint set of nonzero vectors. Suppose that there exsits infinite positive disjoint sequence of nonzero vectors $(x_n)_{n\in \mathbb{N}}$ in $C_T$. Clearly, $|T|(x_n)>0$ for all $n\in \mathbb{N}$. Hence there exsits $y_n\in [-x_n,x_n]$ such that $T(y_n)\ne 0$. Since $(y_n)$ is also a disjoint sequence, so $\big(\dfrac{y_n}{\Vert T(y_n)\Vert}\big)$ is disjoint, but for any $n\in \mathbb{N}$ on has $T\big(\dfrac{y_n}{\Vert T(y_n)\Vert}\big)=1$ and so, $\rightarrow0$ is absurd.
	
Then we prove that $C_T$ is generated by finitely many atoms. Let $X$ be a maximal disjoint family of atoms of $E$ and $A=X\cap C_T$, the linear span $B$ of $A$ is a projection band in $C_T$ since $A$ is a finite set (of atoms). If $B\ne C_T$, denoted by $C_T=B \oplus B^d$, there would exists $0 < x \in C_T$ such that $x\perp B$. Since $x$ is not an atom, there exists $u_1, y$ such that $0 < u_1, y \leq x$ and $u_1 \perp y$. Clearly, $u_1, y \in C_T$. Since $y \perp B$, $y$ is not an atom, and thus there exists $u_2,z$ such that $0 < u_2,z \leq  y$ and $u_2\perp z$. Clearly, $u_2,z \in C_T$. Repeating this process, we obtain an infinite disjoint sequence $(u_n )_{n\in\mathbb{N}}$ in $C_T$, but we have proved that the carrier of $T$ can not contain an infinite disjoint set of nonzero vectors. Hence $B=C_T$.
	
   $(2)-(5)$ It follows form \cite[Corollary~3.6]{GTX:16}, \cite[Lemma~2]{Z:16} and \cite[Lemma~2.3]{W:19} that every disjoint sequence in a Banach lattice $E$ is $uo$-null, $uaw$-null and $uaw^*$-null. According to \cite[Proposition~3.5]{KMT:16}, if $E$ has order continuous norm, every disjoint sequence in $E$ is $un$-null, so we can find that $\big(\dfrac{y_n}{\Vert T(y_n)\Vert}\big)$ is $uo$, $un$, $uaw$ and $uaw^*$-null. The rest of the proof is an application of $(1)$.
	\end{proof}
Let $F=\mathbb{R}$, we have the following result.
\begin{corollary}
	Let $E$ be an atomic Banach lattice, $f$ a nonzero linear functional on $E$ and $(x_\alpha)$ a net in $E$ such that $f(x_\alpha)\rightarrow0$. Then $f$ is the linear combination of the coordinate functionals of finitely many atoms, if one of the following conditions is satisfied.
	\begin{enumerate}
		\item $(x_\alpha)$ is disjoint.
		\item (\cite[Proposition~2.2]{GLX:18}) $x_\alpha\xrightarrow{uo}0$.
		\item (\cite[Corollary~5.4]{KMT:16}) $x_\alpha\xrightarrow{un}0$ and $E$ has order continuous norm.
		\item $x_\alpha\xrightarrow{uaw}0$.
		\item Whenever $E$ is a dual Banach lattices and $x_\alpha\xrightarrow{uaw^*}0$.
	\end{enumerate}
\end{corollary}
According to the above results, we can find that the $uo$-continuous, $un$-continuous, $uaw$-continuous, $uaw^*$-continuous and disjoint continuous functionals only works on finite-dimensional space, hence we study the "bounded" continuous functional for unbounded convergences in Banach lattices.

Let $E$ be a Banach lattice. A linear functional $f$ on $E$ is said to be
\emph{($\sigma$)-order continuous} if $f(x_\alpha)\rightarrow0\big(f(x_n)\rightarrow0\big)$ for any net (sequence) $(x_\alpha)\big((x_n)\big)$ in $E$ that order converges to zero. The set $E^\sim_n$ of all order continuous functionals is said the \emph{order continuous dual} of $E$. In \cite{GLX:18}, a linear functional $f$ on $E$ is said to be \emph{bounded uo-continuous} if $f(x_\alpha)\rightarrow 0$ for any norm bounded $uo$-null net $(x_\alpha)$ in $E$. The set of all bounded $uo$-continuous linear functionals on $E$ will be called the unbounded order dual (uo-dual, for short) of $E$, and will be denoted by $E^\sim_{uo}$. It is natural to consider the other duals for unbounded convergence likes $un$-continuous, $uaw$-continuous and $uaw^*$-continuous.
\begin{definition}
	Let $E$ be a Banach lattice. A bounded linear functional $f$ on $E$ is said to be \emph{bounded d (un, uaw)-continuous} if $f(x_\alpha)\rightarrow 0$ for any norm bounded disjoint ($un$-null, $uaw$-null) net $(x_\alpha)$ in $E$. The set of all bounded $d (un, uaw)$-continuous linear functionals on $E$ will be called the disjoint (unbounded norm, unbounded absolute weak) dual ($d$-dual, $un$-dual and $uaw$-dual, for short) of $E$, and will be denoted by $E^\sim_{d} (E^\sim_{un},E^\sim_{uaw})$.
	
	A bounded linear functional $f$ on $E^\prime$ is said to be \emph{bounded uaw*-continuous} if $f(x_\alpha^\prime)\rightarrow 0$ for any norm bounded $uaw^*$-null net $(x_\alpha^\prime)$ in $E^\prime$. The set of all bounded $uaw^*$-continuous linear functionals on $E^\prime$ will be called the unbounded absolute weak* dual (uaw*-dual, for short) of $E$, and will be denoted by $(E^\prime)^\sim_{uaw^*}$.
\end{definition}
The basic properties of these duals are as follow.
\begin{proposition}\label{123}
	For a Banach lattice $E$, the following holds.
	\begin{enumerate}
		\item $(E^\prime)^\sim_{uaw^*}$ is a closed ideal of $E$;
		\item $E^\sim_{uo}$ is a closed ideal of $E^\sim_n$;
		\item $E^\sim_{uaw}$, $E^\sim_{d}$ and $E^\sim_{un}$ are closed ideals of $E^\prime$.
	\end{enumerate}
\end{proposition}
\begin{proof}
	$(1)$. Since $x_\alpha^\prime\xrightarrow{uaw^*}0\Leftrightarrow |x_\alpha^\prime|\xrightarrow{uaw^*}0$, so we assume $(x_\alpha^\prime)$ is positive. Let $f$ be a bounded $uaw^*$-continuous functional on $E^\prime$. For a net $(x_\alpha^\prime)$ satisfying $|x_\alpha^\prime|\xrightarrow{w^*}0$ in $E^\prime$, clearly, we have $x^\prime_\alpha\xrightarrow{uaw^*}0$ and $f(x_\alpha^\prime)\rightarrow0$. Since $\big(E^\prime,|\sigma|(E^\prime,E)\big)^\prime=E$, hence we have $f\in E$. So $(E^\prime)^\sim_{uaw^*}\subset E$.
	
	It is clear that $(E^\prime)^\sim_{uaw^*}$ is a linear subspace of $E$. We claim that $(E^\prime)^\sim_{uaw^*}$ is a closed ideal of $E$. Since $|f|(x)=\sup\{f(y):|y|\leq x\}$. So for any $\epsilon>0$, there exsits some $\alpha_0$ and a net $(y_\alpha^\prime)\subset E^\prime$ with such that $|f|(x_\alpha^\prime)\leq f(y_\alpha^\prime)+2\epsilon$ whenever $\alpha\geq \alpha_0$. It is clear that $(y_\alpha^\prime)$ is also $uaw^*$-null, hence we have $|f|(x_\alpha^\prime)\rightarrow0$. So $(E^\prime)^\sim_{uaw^*}$ is a sublattice of $E$. For the funcitonals $0\leq g\leq f\in (E^\prime)^\sim_{uaw^*}$. Clearly, $g(x_\alpha^\prime)\leq f(x_\alpha^\prime)\rightarrow0$, hence $g\in (E^\prime)^\sim_{uaw^*}$. So $(E^\prime)^\sim_{uaw^*}$ is a ideal of $E$. Choose some $g\in (E^\prime)^\sim_{uaw^*}$ satisfying $\Vert f-g\Vert<\epsilon$. Since $f(x_\alpha^\prime)=g(x_\alpha^\prime)+(f-g)(x_\alpha^\prime)$, so we have $|f(x_\alpha^\prime)|\leq|g(x_\alpha^\prime)|+|(f-g)(x_\alpha^\prime)|$. Hence $f\in (E^\prime)^\sim_{uaw^*}$. So $(E^\prime)^\sim_{uaw^*}$ is a closed ideal of $E$.
	
	$(2)$ and $(3)$. It is clear that order convergence implies $uo$-convergence, norm convergence implies $un$ and $uaw$-convergence. So we can get that $E^\sim_{uo}$ is a subspace of $E^\sim_n$ and $E^\sim_{uaw}$, $E^\sim_{d}$ and $E^\sim_{un}$ are subspaces of $E^\prime$. The rest of the proof is similar to $(1)$.
	\end{proof}
Recall that the order continuous part $E^a$ of a Banach lattice $E$ is given by $$E^a=\{x\in E:\text{every monotone increasing sequence in $[0,|x|]$ is norm convergent}\}.$$ According to \cite[Corollary~2.3.6]{MN:91}, it is equivalent to $$E^a=\{x\in E:\text{every disjoint sequence in $[0,|x|]$ is norm convergent}\}.$$ A Banach lattice $E$ is said to be \emph{order continuous} whenever $\Vert x_\alpha\Vert\rightarrow0$ for every net $x_\alpha\downarrow0$ in $E$. By \cite[Proposition~2.4.10]{MN:91}, $E^a$ is the largest closed ideal with order continuous norm of $E$.

The following results show some characterizations of the continuity of bounded $uo$, $un$, $uaw$, $uaw^*$ and d-continuous functionals.
\begin{theorem}\label{uaw*,uo,uaw,un-functional=d-dual}
Let $E$ be a Banach lattice and $\mathscr{F}_x$ a functional on $E^\prime$ for any $x\in E$, the following conditions are equivalent.
	\begin{enumerate}
		\item $\mathscr{F}_x\in (E^\prime)^\sim_{uaw^*}$.
		\item $\mathscr{F}_x(x_n^\prime)\rightarrow0$ for any bounded $uaw^*$-null sequence $(x_n^\prime)$ in $E^\prime$.
		\item $\mathscr{F}_x\in (E^\prime)^\sim_{uo}$.
		\item $\mathscr{F}_x(x_n^\prime)\rightarrow0$ for any bounded $uo$-null sequence $(x_n^\prime)$ in $E^\prime$.
		\item $\mathscr{F}_x\in (E^\prime)^\sim_{uaw}$.
		\item $\mathscr{F}_x(x_n^\prime)\rightarrow0$ for any bounded $uaw$-null sequence $(x_n^\prime)$ in $E^\prime$.
		\item $\mathscr{F}_x\in (E^\prime)^\sim_{d}$.
		\item $\mathscr{F}_x(x_n^\prime)\rightarrow0$ for any bounded disjoint sequence $(x_n^\prime)$ in $E^\prime$.
		\item Every disjoint sequence in $[0,|x|]$ is norm convergent to zero.
		
		In addition, if $E^\prime$ has order continuous norm, these conditions are equivalent to
		
		\item $\mathscr{F}_x\in (E^\prime)^\sim_{un}$.
		\item $\mathscr{F}_x(x_n^\prime)\rightarrow0$ for any bounded $un$-null sequence $(x_n^\prime)$ in $E^\prime$.	
	\end{enumerate}
\end{theorem}
\begin{proof}
	
	$(1)\Rightarrow(2)$, $(3)\Rightarrow(4)$, $(5)\Rightarrow(6)$ and $(7)\Rightarrow(8)$ are obvious.  $(1)\Rightarrow(5)$ and $(2)\Rightarrow(6)$ by $uaw$-convergence implies $uaw^*$-convergence.
	
	$(1)\Rightarrow(3)$. Let $(x_\alpha^\prime)$ be a $uo$-null in $E^\prime$ and $\mathscr{F}_x\in (E^\prime)^\sim_{uaw^*}$. Hence $|x_\alpha^\prime|\wedge u^\prime\leq y_\beta^\prime\downarrow 0$ in $E^\prime$ for all $u^\prime\in 
	E^\prime_+$. Clearly, for a positive element $x\in E$, we have $(|x_\alpha^\prime|\wedge u^\prime)(x)\leq (y_\beta^\prime)(x)\rightarrow 0$. Since $\mathscr{F}_x\in (E^\prime)^\sim_{uaw^*}$, therefore $\mathscr{F}_x\in (E^\prime)^\sim_{uo}$. $(2)\Rightarrow(4)$ is similar.
	
	$(1)\Rightarrow(7)$, $(2)\Rightarrow(8)$, $(5)\Rightarrow(7)$ and $(6)\Rightarrow(8)$. It follows from \cite[Lemma~2]{Z:16} and \cite[Lemma~2.3]{W:19} that every disjoint net is $uaw$-null and $uaw^*$-null. According to \cite[Corollary~3.6]{GTX:16}, we have $(4)\Rightarrow(8)$.
	
	$(8)\Rightarrow(1)$. Since $(x_n^\prime)$ is a disjoint sequence, hence for a sequence $(y_n^\prime)$ statisfying $\{y_n^\prime\in [-|x_n^\prime|,|x_n^\prime|]\}$ is also disjoint. Theorefore $\sup_{y_n^\prime\in [-|x_n^\prime|,|x_n^\prime|]} |\mathscr{F}_x(y_n^\prime)|=|\mathscr{F}_x|(|x_n^\prime|)\rightarrow0$ for any disjoint sequence $(x_n^\prime)$ in $B_{E^\prime}$. Using \cite[Theorem~4.36]{AB:06}, let the seminorm be $|\mathscr{F}_x|(|\cdot|)$, $T$ be idential operator and bounded solid set be $B_{E^\prime}$. So we have that, for any $\epsilon>0$, there exists $u^\prime\in E^\prime_+$ such that	$$\sup_{x^\prime\in B_{E^\prime}}|\mathscr{F}_x|(|x^\prime|-|x^\prime|\wedge u^\prime)=\sup_{x\in B_{E^\prime}}|\mathscr{F}_x|\big((|x^\prime|-u)^+\big)<\epsilon.$$
	
	For a $uaw^*$-null net $(x_\alpha^\prime)\subset B_{E^\prime}$, that is $|x_\alpha^\prime|\wedge u^\prime\xrightarrow{w^*}0$. Hence $|\mathscr{F}_x|(|x_\alpha^\prime|\wedge u^\prime)\rightarrow0$. Theorefore $|\mathscr{F}_x(x_\alpha^\prime)|\leq|\mathscr{F}_x|(|x_\alpha^\prime|)\rightarrow0$.
	
	$(8)\Leftrightarrow(9)$. According to \cite[Corollary~2.3.3]{MN:91}, let $A=[-|x|,|x|]$ and $B=B_{E^\prime}$. Every disjoint sequence in $[0,|x|]$ is norm convergent to zero iff every disjoint sequence in $[-|x|,|x|]$ is uniform convergences to zero on $B$. Since $(x_n^\prime)$ is disjoint iff $(|x_n^\prime|)$ is disjoint and $\mathscr{F}_x(|x_n^\prime|)=\sup_{g\in [-|x|,|x|]}|g(x_n^\prime)|$, so $\mathscr{F}_x(x_n^\prime)\rightarrow0$ for any bounded disjoint sequence $(x_n^\prime)$ in $E^\prime$ iff $(x_n^\prime)$ is uniform convergence to zero on $A$. We have the result.
	
	$(5)\Leftrightarrow(11)$. Suppose now that $E^\prime$ is order continuous. According to \cite[Theorem~2.4]{W:19}, $uaw$ and $uaw^*$-topologies coincide with the
	$un$-topology. The result can be easily verified.
	
\end{proof}

\begin{theorem}\label{uo,uaw,un-functional=d}
(extension of \cite[Theorem~2.3]{GLX:18})	Let $E$ be a Banach lattice, for any $f\in E^\sim_n$, the following conditions are equivalent.
	\begin{enumerate}
		\item $f\in E^\sim_{uo}$.
		\item $f(x_n)\rightarrow0$ for any bounded $uo$-null sequence $(x_n)$ in $E$.
		\item $f\in E^\sim_{uaw}$.
		\item $f(x_n)\rightarrow0$ for any bounded $uaw$-null sequence $(x_n)$ in $E$.
		\item $f\in E^\sim_{d}$.
		\item $f(x_n)\rightarrow0$ for any bounded disjoint sequence $(x_n)$ in $E$.
		\item Every disjoint sequence in $[0,|f|]$ is norm convergent to zero.
		
		In addition, if $E$ has order continuous norm, these conditions are equivalent to
		
		\item $f\in E^\sim_{un}$.
		
		\item $f(x_n)\rightarrow0$ for any bounded $un$-null sequence $(x_n)$ in $E$.
	\end{enumerate}
\end{theorem}
\begin{proof}
	$(1)\Leftrightarrow(2)\Leftrightarrow(6)\Leftrightarrow(7)$. By \cite[Theorem~2.3]{GLX:18}.
	
	$(3)\Leftrightarrow(9)$ is similar to $(5)\Leftrightarrow(11)$ of Theorem \ref{uaw*,uo,uaw,un-functional=d-dual}.
	
	$(3)\Rightarrow(4)$ and $(5)\Rightarrow(6)$ are obvious. $(3)\Rightarrow(5)$ and $(4)\Rightarrow(6)$ by every disjoint net is $uaw$-null.
	
	$(6)\Rightarrow(3)$.  Using \cite[Theorem~4.36]{AB:06}, the proof is similar to $(3)\Rightarrow(1)$ of \cite[Theorem~2.3]{GLX:18} and $(8)\Rightarrow(1)$ of Theorem \ref{uaw*,uo,uaw,un-functional=d-dual}.
	\end{proof}
Similarly, we have,
\begin{theorem}\label{uaw,un-functional=d}
	Let $E$ be a Banach lattice, for any $f\in E^\prime$, the following conditions are equivalent.
	\begin{enumerate}
		\item $f\in E^\sim_{uaw}$.
		\item $f(x_n)\rightarrow0$ for any bounded $uaw$-null sequence $(x_n)$ in $E$.
		\item $f\in E^\sim_{d}$.
		\item $f(x_n)\rightarrow0$ for any bounded disjoint sequence $(x_n)$ in $E$.
		\item Every disjoint sequence in $[0,|f|]$ is norm convergent to zero.
		
		In addition, if $E$ has order continuous norm, these conditions are equivalent to
		
		\item $f\in E^\sim_{un}$.
		
		\item $f(x_n)\rightarrow0$ for any bounded $un$-null sequence $(x_n)$ in $E$.
	\end{enumerate}
\end{theorem}
Using the above results, we obtain the exact form of these duals.
\begin{theorem}\label{uo-un-uaw-uaw^*-d=^a}
	Let $E$ be a Banach lattice, The following relations hold.
	\begin{enumerate}
		\item\label{11} $E^\sim_{uo}=(E^\sim_n)^a\subset E^\sim_{uaw}=E^\sim_d=(E^\prime)^a\subset  E^\sim_{un}\subset E^\prime$.
		\item\label{12} $(E^\prime)^\sim_{uaw^*}=E^a\subset (E^\prime)^\sim_{uo}=\big((E^\prime)^\sim_n)^a\subset (E^\prime)^\sim_{uaw}=(E^\prime)^\sim_{d}=(E^{\prime\prime})^a\subset (E^\prime)^\sim_{un}\subset E^{\prime\prime}$.
	\end{enumerate}
\end{theorem}
\begin{proof}
	According to Proposition \ref{123}, we have $(E^\prime)^\sim_{uaw^*}\subset E$, $E^\sim_{uo}\subset E^\sim_n$, $E^\sim_{uaw}\subset E^\prime$, $E^\sim_{d}\subset E^\prime$ and $E^\sim_{un}\subset E^\prime$. It follows from Theorem \ref{uaw*,uo,uaw,un-functional=d-dual}, \ref{uo,uaw,un-functional=d} and \ref{uaw,un-functional=d} that these duals are the order continuous part, therefore we have the result.
	\end{proof}
The following example shows the bounded duals for unbounded convergence in classical Banach lattices.
\begin{example}\label{d-example}
	\begin{align*}
		&(c_0)^\sim_{uo}=(c_0)^\sim_{uaw}=(c_0)^\sim_d=(c_0)^\sim_{un}=(l_1)^a=l_1,\\&(l_1)^\sim_{uaw^*}=(l_1)^\sim_{uo}=(l_1)^\sim_{uaw}=(l_1)^\sim_d=(l_\infty)^a=(c_0)^a=c_0,\\
		&(l_\infty)^\sim_{uaw^*}=(l_\infty)^\sim_{uo}=(l_1)^a=l_1, \\
		&(l_\infty)^\sim_{uaw}=(l_\infty)^\sim_{d}=(l_\infty)^\sim_{un}=ba(2^\mathbb{N}),\\
		&(L_1[0,1])^\sim_{uo}=(L_1[0,1])^\sim_{uaw}=(L_1[0,1])^\sim_{d}=(L_\infty[0,1])^a=\{0\},\\
		&(L_\infty[0,1])^\sim_{uaw^*}=(L_\infty[0,1])^\sim_{uo}=(L_1[0,1])^a=L_1[0,1],\\
		&(L_\infty[0,1])^\sim_{uaw}=(L_\infty[0,1])^\sim_{d}=(L_\infty[0,1])^\sim_{un}=(ba[0,1])^a=ba[0,1],\\
		&(C[0,1])^\sim_{uo}=(\{0\})^a=\{0\},\\
		&(C[0,1])^\sim_{uaw}=(C[0,1])^\sim_{d}=(C[0,1])^\sim_{un}=rca[0,1].
	\end{align*}
\end{example}

As an application of these results, we conclude the paper with characterizations of the order continuity and reflexivity of Banach lattices.
\begin{theorem}\label{}
	For a Banach lattice $E$, the following holds.
	\begin{enumerate}
		\item $E$ has order continuous norm iff $(E^\prime)^\sim_{uaw^*}=E$;
		\item $E^\prime$ has order continuous norm iff $E^\sim_{uaw}=E^\sim_d=E^\sim_{un}=E^\prime$;
		\item (extension of \cite[Theorem~5]{W:77}) $E$ and $E^\prime$ are order continuous iff $E^\sim_{uo}=E^\sim_{uaw}=E^\sim_d=E^\sim_{un}=E^\prime$;
		\item $E$ is reflexive iff $(E^\prime)^\sim_{uaw^*}=(E^\prime)^\sim_{uo}=(E^\prime)^\sim_{uaw}=(E^\prime)^\sim_{d}=(E^\prime)^\sim_{un}=E^{\prime\prime}$.
	\end{enumerate}
\end{theorem}
\begin{proof}
$(1)$. $E$ is order continuous iff $E^a=E$. It follows from Theorem \ref{uo-un-uaw-uaw^*-d=^a}\eqref{12} that $E$ is order continuous iff $(E^\prime)^\sim_{uaw^*}=E^a=E$.
	
$(2)$. $E^\prime$ is order continuous iff $(E^\prime)^a=E^\prime$. According to Theorem \ref{uo-un-uaw-uaw^*-d=^a}\eqref{11}, we have $E^\sim_{uaw}=E^\sim_d=(E^\prime)^a$, therefore $E^\prime$ is order continuous iff $E^\sim_{uaw}=E^\sim_d=E^\sim_{un}=(E^\prime)^a=E^\prime$.
	
$(3)$. $E$ and $E^\prime$ are order continuous iff $(E^\sim_n)^a=E^\prime$. Since $E^\sim_{uo}=(E^\sim_n)^a$, therefore $E$ and $E^\prime$ are order continuous iff $E^\sim_{uo}=E^\sim_{uaw}=E^\sim_d=E^\sim_{un}=(E^\sim_n)^a=(E^\prime)^a=E^\prime$.
	
$(4)$. $E$ is reflexive iff $E^a=E=E^{\prime\prime}$. Hence $E$ is reflexive iff  $(E^\prime)^\sim_{uaw^*}=(E^\prime)^\sim_{uo}=(E^\prime)^\sim_{uaw}=(E^\prime)^\sim_{d}=(E^\prime)^\sim_{un}=E^a=E=E^{\prime\prime}$ by $(E^\prime)^\sim_{uaw^*}=E^a$.
	\end{proof}
So far, we still do not know what the $E^\sim_{un}$ is. So, we state the problem in here.
\begin{problem}
	What is the exact form of $E^\sim_{un}$?
\end{problem}
\noindent \textbf{Acknowledgement.} The research is supported by National Natural Science Foundation of China (NSFC:51875483).

\end{document}